\documentclass[a4paper,12pt]{article}
\usepackage[centertags]{amsmath}
\usepackage{amsfonts}
\usepackage{amssymb}
\usepackage{amsthm}
\usepackage{dsfont}
\usepackage{mathrsfs}
\usepackage{wasysym}
\usepackage{esint}
\usepackage[draft]{optional}
\usepackage[usenames]{color}
\usepackage{url}
\usepackage{hyperref}
\theoremstyle{plain}
\newtheorem{lemma}{Lemma}[section]
\newtheorem{theorem}[lemma]{Theorem}
\newtheorem{prop}[lemma]{Proposition}
\newtheorem{cor}[lemma]{Corollary}

\theoremstyle{definition}

\newtheorem{remark}[lemma]{Remark}

\newcommand{\twomat}[4]{ \left( \begin{array}{cc}
#1 & #2 \\
#3 & #4
\end{array} \right)}

\newcommand{\cvec}[2] {\left( \begin{array}{c}
#1 \\
#2
\end{array}\right)}

\newcommand{\twopartdef}[4]
{
	\left\{
		\begin{array}{ll}
			#1 & \mbox{if  } #2 \bigskip \\
			#3 & \mbox{if  } #4
		\end{array}
	\right.
}

\newcommand{\R}{\mathbb{R}}

\newcommand{\Z}{\mathbb{Z}}

\newcommand{\h}{\mathcal{H}}

\newcommand{\Ca}{\mathcal{C}}

\newcommand{\M}{M}
\newcommand{\tu}{\rightarrow}

\newcommand{\emb}{\hookrightarrow}
\newcommand{\cemb}{\subset \subset}
\newcommand{\In}{\subset}
\newcommand{\Om}{\Omega}
\newcommand{\om}{\omega}
\newcommand{\dl}{{\delta}}
\newcommand{\Dl}{{\Delta}}
\newcommand{\ta}{{\theta}}
\newcommand{\al}{{\alpha}}

\newcommand{\ed}{{\rm d}}
\newcommand{\id}{\,\,\ed}
\newcommand{\dv}{{\rm d}^{\ast}}
\newcommand{\D}{{\nabla}}

\newcommand{\ti}[1]{{\tilde{#1}}}

\newcommand{\fr}[2]{\frac{#1}{#2}}
\newcommand{\sm}{{\setminus}}

\newcommand{\Nn}{\mathcal{N}}
\newcommand{\T}{{\rm T}}

\def\Xint#1{\mathchoice
{\XXint\displaystyle\textstyle{#1}}%
{\XXint\textstyle\scriptstyle{#1}}%
{\XXint\scriptstyle\scriptscriptstyle{#1}}%
{\XXint\scriptscriptstyle\scriptscriptstyle{#1}}%
\!\int}
\def\XXint#1#2#3{{\setbox0=\hbox{$#1{#2#3}{\int}$}
\vcenter{\hbox{$#2#3$}}\kern-.5\wd0}}

\def\dint{\Xint-}

\begin{document}
 \title{Higher integrability for solutions to a system of critical elliptic PDE}
\author{Ben Sharp \footnote{{\sc Department of Mathematics, Huxley Building, 180 Queen's Gate, South Kensington Campus, Imperial College London, SW7 2AZ, UK. Email:  benjamin.sharp@imperial.ac.uk} Address at time of writing: {\sc mathematics institute, university of Warwick, Coventry, CV4 7AL,
UK}}}
\maketitle 

\begin{abstract}
We give new estimates for a critical elliptic system introduced by Rivi\`ere-Struwe which generalises PDE solved by (almost) harmonic maps from a Euclidean ball $B_1 \In \R^n$ into closed Riemannian manifolds $\Nn\emb \R^m$. Solutions $u:B_1\to \R^m$ take the form 
$$-\Dl u^i = \Om^i_j.\D u^j  $$ 
where $\Om$ maps into antisymmetric $m\times m$ matrices with entries in $\R^n$.  Here $\Om$ and $\D u$ belong to a Morrey space which makes the PDE critical from a regularity perspective. We use the Coulomb frame method employed by Rivi\`ere-Struwe along with the H\"older regularity already acquired therein, coupled with an extension of a Riesz potential estimate,  in order to improve the known regularity and estimates for solutions $u$. These methods apply when $n=2$ thereby re-proving the full regularity in this case using Coulomb gauge methods. Moreover they lead to a self contained proof of the local regularity of stationary harmonic maps in high dimension.
\end{abstract}

\section{Introduction}

The system of PDE under consideration here, introduced by Rivi\`ere \cite{riviere_inventiones} and Rivi\`ere-Struwe \cite{riviere_struwe}, generalises many PDE present in the study of critical geometric problems. It has `hidden' regularity properties by virtue of the anti-symmetry of the term `$\Om$', and one can both recover and somewhat improve the regularity theory for the aforementioned  geometric PDE via a much more general theory. For instance, consider weakly harmonic maps $u\in W^{1,2}(B_1,\Nn)$ from the Euclidean ball $B_1\In \R^n$ into a closed Riemannian manifold $\Nn\emb \R^m$. These are defined to be critical points with respect to outer variations of 
\begin{equation*}
E(u):= \fr12 \int_{B_1} |\D u|^2,
\end{equation*}
meaning that maps $u\in W^{1,2}(B_1,\Nn)$ form the natural admissible function space for this problem. It is known that $u$ weakly solves a PDE of the form 
\begin{equation}\label{hm}
-\Dl u = A(u)(\D u, \D u)
\end{equation}
where $A$ is the second fundamental form of the embedding $\Nn\emb\R^m$.  Since $\D u\in L^2$ one can estimate the right hand side of \eqref{hm} in $L^1$ - $A$ is a bounded bilinear form. From here the $L^1$ theory for the Laplacian is not enough to begin any kind of bootstrapping argument unless $n=1$. For instance, when $n=2$ standard theory gives estimates on $\D u$ in the Lorentz space $L^{2,\infty}$, a space slightly larger than $L^2$. For $n\geq 3$ the analogous estimates give control on $\D u$ in $L^{\fr{n}{n-1}, \infty}$ which is much worse than being in $L^2$. Indeed the regularity theory for weakly harmonic maps in higher dimensions is doomed to fail since there exist `nowhere continuous' weakly harmonic maps whenever $n\geq 3$ \cite{riviere_singular}. In two dimensions however, all weakly harmonic maps are smooth, which one can prove using a moving frame technique due to H\'elein \cite{helein_regularity}. The same technique can be used to prove partial regularity in higher dimensions (\cite{evans_spheres} for $\Nn = S^{m-1}$ and \cite{bethuel_singset} for general closed $\Nn$) where one can prove smoothness under the condition that $\|\D u\|_{L^2}$ decays sufficiently quickly on small balls - i.e. it lies in a Morrey space. Roughly speaking, these techniques use the fact that under an appropriate frame the `bad' $L^1$ terms can be replaced by better $\h^1$ terms. The space $\h^1$ is a Hardy space - functions which are in $L^1$ but have additional cancellation properties - see below for a definition.

Given $g\in L_{loc}^1(O)$ and $O\In\R^n$, we say that $g$ lies in the Morrey space $M^{p,\beta}(O)$ if the following function 
\begin{equation*}
M_{\beta}[|g|^p](x): = \sup_r r^{-\beta} \int_{B_r(x)\cap O} |g|^p 
\end{equation*}
is bounded; with a norm given by $\|g\|_{M^{p,\beta}(O)}^p:= \|M_{\beta}[|g|^p]\|_{L^{\infty}(O)}$. See appendix \ref{Morrey} for more details. We have followed \cite{giaquinta} in terms of our choice of indices, though we use $M^{p,\beta}$ rather than $L^{p,\beta}$ so as to avoid confusion with Lorentz spaces. 

From now on we will denote $M_n[g] = M[g]$ giving the usual maximal function. We now define a more refined version for $g:\R^n\to \R$:  let $\phi \in C_c^{\infty}(B_1)$ with $\|\D \phi\|_{L^{\infty}}\leq 1$ and $\phi_t(x): = t^{-n}\phi(\fr{x}{t})$. We set 
\begin{equation*}
g_{\ast}(x) = \sup_{0<t<\infty} |\phi_t \ast g (x)|
\end{equation*}
and say that $g\in \h^1(\R^n)$ if $g_\ast\in L^1(\R^n)$ with norm $\|g\|_{\h^1(\R^n)} = \|g_{\ast}\|_{L^1(\R^n)}$. It can be shown that this norm is independent of the choice of $\phi$ up to a constant - see \cite{semmes_primer}.  The related local Hardy space $h^1$ is defined similarly; we require that 
\begin{equation*}
g_{\ti\ast}(x) = \sup_{0<t<1} |\phi_t \ast g (x)|\in L^1
\end{equation*}
with the obvious associated norm - we clearly have the continuous embeddings $\h^1(\R^n)\emb h^1(\R^n)  \hookrightarrow L^1(\R^n)$. Recall that $M[g] \in L^p$ iff $g\in L^p$ when $p>1$, and that $M[g]\in L^1$ iff $g\equiv 0$. Therefore $g\in \h^1$ iff the more refined version of the 
maximal function $g_{\ast}$, is integrable. We see $\h^1$ as being a replacement for $L^1$ in this sense. Most importantly for us the Hardy spaces behave well with respect to singular integral operators and therefore in elliptic regularity theory, again in contrast to the $L^1$ theory. 

As mentioned above, when $n\geq 3$, there exist nowhere regular weakly harmonic maps. If however, we consider weakly stationary harmonic maps - critical points of $E$ with respect to outer {\emph{and}} inner variations - then we obtain that
$$\rho(r) = r^{2-n}\int_{B_r(x)} |\D u|^2$$
is monotone increasing for any $x\in B_1$, $r<1-|x|$ see \cite{simon_regularity}. In particular this gives $\D u \in M_{loc}^{2,n-2}(B_1)$. Moreover given any $\epsilon >0$ there exists a closed set $S$ such that for any $z\notin S$ we can find $r_0 = r_0(z)>0$ with $\sup_{r<r_0} r^{2-n}\int_{B_r(z)} |\D u|^2 \leq\epsilon$, and $H^{n-2}(S) = 0$ where $H$ is the Hausdorff measure - see for instance \cite[Chapter IV]{giaquinta}. It is an easy exercise to check that by the monotonicity of $\rho$ we can equivalently say that for any $z\notin S$ there is some $r_0(z)>0$ with $\|\D u\|_{M^{2,n-2}(B_{r_0}(z))} \leq \epsilon$.

The $\epsilon$ regularity results of Evans-Bethuel for $n\geq 3$ state that there is an $\epsilon>0$ such that whenever a weakly stationary harmonic map $u$ satisfies $\rho(R)\leq \epsilon$ for some $x\in B_1$ and $0<R<1-|x|$, then $u$ is smooth on $B_{\fr{R}2}(x)$. Thus we can conclude that weakly stationary harmonic maps are smooth away from a singular set $S$ with $H^{n-2}(S) = 0$.

In \cite{riviere_inventiones} and \cite{riviere_struwe} it is shown that by adding a null term to the harmonic map equation \eqref{hm} one has that $u:B_1 \to \Nn\emb\R^m$ solves\begin{equation}\label{f}
-\Dl u = \Om^{A} .\D u,
\end{equation}
for some $\Om^{A}: B_1 \to so(m)\otimes \wedge^1 \R^n$ and the product $[\Om^{A}.\D u]^i = \sum_j\langle (\Om^{A})^i_j, \D u^j \rangle$ is an inner product of forms coupled with matrix multiplication. As we shall see, the anti-symmetry of $\Om^{A}$ is crucial for obtaining any higher regularity of this PDE. In this paper we will study this equation under very natural and critical conditions on arbitrary $\Om$ and $u$ solving \eqref{f}.

In the setting of harmonic maps one also knows that there is some $C=C(\Nn)>0$ such that $|\Om^{A}(x)|\leq C|\D u(x)|$. Therefore the natural conditions for $\Om$ and $\D u$ are that they lie in $M^{2,n-2}$ as in \cite{riviere_struwe}, where they prove that, when $\|\Om\|_{M^{2,n-2}(B_1)}$ is small enough, $u$ is H\"older continuous, Theorem \ref{old}.

The starting point for us will be to consider arbitrary $u\in W^{1,2}(B_1,\R^m)$ weakly solving 
\begin{equation}\label{ff}
-\Dl u = \Om.\D u + f
\end{equation}
on the unit ball $B_1\In \R^n$, for some $\Om \in M^{2,n-2}(B_1,so(m)\otimes \wedge^1\R^n)$ and $f\in L^p$ under the assumptions that $\D u \in M^{2,n-2}(B_1, \R^m \otimes \wedge^1 \R^n)$ and $\fr{n}2 <p<n$. The question which will be addressed is whether or not one has higher integrability of $\D u$ (beyond $L^2$) for general systems \eqref{ff}. In \cite{riviere_struwe}, \cite{rupflin} and \cite{schikorra_frames} the H\"older regularity is obtained by improving decay on $\|\D u\|_{L^s(B_r)}$ for $s<2$. Therefore any higher integrability does not follow directly from such a result. Moreover it is difficult to see how one might boot-strap this information back into the PDE to improve the regularity. 
 
Estimating the right hand side of \eqref{ff} using H\"older's inequality leaves us with $\Dl u \in \M^{1,n-2}$ (=$L^1$ when $n=2$), and the best we can do using singular integral estimates is to conclude that $\D u \in \M^{(2,\infty),n-2}$ (= $L^{(2,\infty)}$ when $n=2$). See appendix \ref{Morrey} for definitions and results if necessary. These spaces are slightly worse than the spaces we started with, therefore we have lost some information and bootstrapping fails. The antisymmetry condition on $\Om$ is therefore key to unlocking the hidden regularity of this system as first noticed by Rivi\`ere \cite{riviere_inventiones} - it is known that by dropping the antisymmetry of $\Om$ the regularity theory fails. 

We interpret $\Om$ as connection forms for the trivial bundle $B_1 \times \R^m$, for which $\ed u \in \Gamma(\T^{\ast}B_1 \otimes \R^m)$ and \eqref{ff} reads $\dv_{\Om}(\ed u) = f$ where $\dv_{\Om}$ is the induced covariant divergence given by $\Om$ (the formal adjoint of the covariant exterior derivative, $\dv_{\Om}(\ed u) = \dv \ed u  - \ast [\Om \wedge \ast \ed u]$).
With this more geometric setting it is possible to talk of changes of frame or gauge, and crucially the antisymmetry of $\Om$ means that any change of gauge lies in the orthogonal group and carries a natural $L^{\infty}$ bound. A change of gauge here is a purely local affair and consists of a map $P:B_1 \tu SO(m)$ in which we can express $\ed u$, $\Om$ and therefore our PDE. Under this change of gauge the new connection forms $\Om_P$ look like $\Om_P = P^{-1}\ed P + P^{-1}\Om P$ and we have that $\dv_\Om (\ed u) = P (\dv_{\Om_P} (P^{-1}\ed u)).$ Therefore, solutions to \eqref{ff} are also solutions to 
$$\dv_{\Om_P} (P^{-1}\ed u) = \dv (P^{-1}\ed u) - \ast[\Om_P\wedge \ast P^{-1}\ed u] = P^{-1}f$$
for any such $P$.

\subsection{Two-dimensional domains}

Given that the problem here is concerned with improving regularity, the game has been to find a gauge that forces this equation to exhibit nice regularity properties. When $n=2$ it was shown in \cite{riviere_inventiones} that we can change the gauge such that the term $\Om.\D u$ is effectively replaced by a Jacobian determinant. Thus we may use Hardy space methods \cite{clms} or Wente-type estimates \cite{wente} to  improve our situation. It was shown in \cite{riviere_inventiones} that the most suitable gauge transform is a small perturbation of the Coulomb gauge (or Uhlenbeck gauge) and in fact it is necessary for the gauge to leave the orthogonal group; moreover these methods allow us to write the PDE as a conservation law. 

Solutions with $f \equiv 0$ are shown (in \cite{riviere_inventiones}) to describe critical points of conformally invariant elliptic Lagrangians under some natural growth assumptions and for appropriate $\Om$. 
In particular when $f \equiv 0$, \eqref{f} describes harmonic maps and prescribed mean curvature equations from Riemannian surfaces into closed, $C^2$ Riemannian manifolds $N \emb \R^m$ isometrically embedded in some Euclidean space.

This PDE has subsequently been studied from a regularity and compactness perspective, see for instance \cite{riviere_laurain}, \cite{LZ}, \cite{schikorra_boundary}, \cite{schikorra_frames}, \cite{Sh_To}. In \cite{Sh_To}, it is shown that general solutions to \eqref{f} (when $n=2$) are in $W^{2,s}_{loc}$ for all $s<2$ by means of a Morrey estimate, and we see that Theorem \ref{Main} and Proposition \ref{main} are the analogues of \cite[Theorem 1.1]{Sh_To} and \cite[Lemma 7.3]{Sh_To} in the higher dimensional setting.  

\subsection{Higher dimensional domains}

For $n\geq 2$ Rivi\`ere-Struwe \cite{riviere_struwe} showed that we can find a Coulomb gauge in the Morrey space setting (see appendix \ref{hodge}), and that this is enough to conclude partial regularity for general solutions. Again this comes down to the appearance of terms that lie in the Hardy space $\h^1$. It is shown that solutions to \eqref{f} describe harmonic (and almost harmonic) maps from the Euclidean ball into arbitrary Riemannian manifolds.
As outlined in \cite{riviere_struwe} it would be difficult to carry out the same techniques when $n \geq 3$ as in the case $n=2$, however Laura Keller \cite{keller} has shown that when $\Om$ and $\D u$ lie in a (slightly more restrictive) Besov-Morrey space, then the methods as in the two dimensional case apply. 

The regularity obtained in \cite{riviere_struwe} and \cite{rupflin} is as follows (see also \cite{schikorra_frames}): 

\begin{theorem}\label{old}
Let $u$, $\Om$ and $f$ be as in \eqref{ff}. There exists $\epsilon = \epsilon(n,m,p)$ such that whenever $\|\Om\|_{\M^{2,n-2}(B_1)}^2 \leq \epsilon$ then $u \in C^{0,\gamma}_{loc}$ where $\gamma = 2 -\fr np \in (0,1)$. 
\end{theorem}  
The optimal H\"older regularity was shown in \cite{rupflin} along with an estimate. To see the optimality just consider the case $\Om \equiv 0$; we have that $u \in W^{2,p}_{loc} \emb C^{0,2-\fr np}_{loc}$ when $\fr n2 < p < n$ by Calderon-Zygmund theory and Morrey estimates. 

As stated in \cite{riviere_struwe}, this theorem allows us to extend the regularity theory for stationary harmonic maps from the Euclidean ball into closed $C^2$ Riemannian manifolds immersed in some Euclidean space. More precisely it is possible to show that any weakly stationary harmonic map is smooth away from a singular set $S$ with $H^{n-2}(S)=0$. This follows from a classical theorem stating that continuous weakly harmonic maps are smooth.

The methods for the higher dimensional theory have also been used in the study of Dirac harmonic maps (e.g. \cite{WX} \cite{CJWZ}) and weakly harmonic maps into pseudo-Riemannian manifolds \cite{Zhu}.

\subsection{Statement of Results} 

In this paper we will show improved regularity along with a new estimate. In order to get this estimate we use the Coulomb gauge obtained in \cite{riviere_struwe}, Theorem \ref{old} and we crucially require an extension of a result of Adams \cite{adams_riesz} Theorem \ref{Adams}.

\begin{theorem}\label{Main}
For $n\geq 2$ let $u\in W^{1,2}(B_1,\R^m)$ with $\D u \in \M^{2,n-2}(B_1,\R^m)$, $\Om \in \M^{2,n-2}(B_1,so(m)\otimes \wedge^1 \R^n)$ and $f \in L^p(B_1)$, for $\fr{n}{2}<p<n$, weakly solve
\begin{equation*}
-\Dl u = \Om.\D u + f.
\end{equation*} 
Then for any $U\cemb B_1$ there exist $\epsilon = \epsilon(n,m,p)>0$ and $C=C(n,m,p,U)>0$ such that whenever $\|\Om\|_{\M^{2,n-2}(B_1)} \leq \epsilon$ we have 
\begin{equation*}
\|\D^2 u\|_{\M^{\fr{2p}{n}, n-2}(U)} + \|\D u\|_{\M^{\fr{2p}{n-p},n-2}(U)}  \leq C (\|u\|_{L^1(B_1)} + \|f\|_{L^p(B_1)}).
\end{equation*}
\end{theorem}

We see that this generalises \cite{Sh_To} to higher dimensions, and that if $\D u \in \M^{\fr{2p}{n-p},n-2}$ then $ u \in C^{0,\gamma}$ with $\gamma$ as in Theorem \ref{old}. Moreover, with $p$ in the above range we have $\fr{2p}{n} >1$ and $\fr{2p}{n-p}>2$ - i.e. we have obtained integrability above the critical level. An interesting question here is whether the integrability of $\D u$ can be improved further when $n\geq 3$. One might expect that we should have estimates on $\D u$ in $L^{\fr{np}{n-p}}$ (consider the case $\Om \equiv 0$). Clearly the case $n=2$ is no problem as this gives the optimal regularity expected, moreover we have found solutions with $f\equiv  0$ that are not in $W^{2,2}$ or even $W^{2,(2,\infty)}$. We leave it to the reader to check that defining $u:B_{e^{-1}} \to \R^2$, by $u(x,y): = (\log r)^2\cvec{x}{y}$ with $r = (x^2+y^2)^{\fr12}$ and
$$\Om = \fr{2(1+2\log r)}{(r\log r)^2}\twomat{0}{1}{-1}{0} (-y \ed x + x\ed y)$$gives the desired example - see \cite[Chapter 4.3]{thesis}. Thus we cannot expect that $\D u \in L^{\infty}$ or even $\D u \in BMO$ in general. This also explains the range of $f$ that we consider.

\begin{remark}
Since writing this paper there have been a few generalisations and extensions of Theorem \ref{Main}, namely Roger Moser \cite{moser_regularity} has studied almost harmonic maps: functions $u:B_1\to \Nn\emb \R^m$ solving
$$-\Dl u = A(u)(\D u, \D u) + f$$ with $f\in L^p$, $p>1$ and has obtained an appropriate $W^{2,p}\cap W^{1,2p}$ estimate on such $u$ when $\|\D u\|_{M^{2,n-2}}$ is small. One would expect this result to hold for the system under consideration in this paper, under the further condition that $|\Om|\leq C|\D u|$.  

Armin Schikorra \cite{schikorra_higher} has proved that an analogue of Theorem \ref{Main} holds for more general systems involving non-local operators, using  different techniques.

Finally, in a joint work with Miaomiao Zhu \cite{ben_zhu}, the author has studied the boundary regularity problem for similar systems in order to obtain a full regularity theory in the free boundary problem for Dirac harmonic maps from surfaces.   
\end{remark}

An easy consequence of Theorem \ref{Main} is the following

\begin{cor}\label{highint}
Let $u$ and $\Om$ be as in Theorem \ref{Main} with $f \equiv 0$. 
For any $q < \infty$ and $U\cemb B_1$, setting $s=\fr{2q}{2+q}<2$, there exist  $\epsilon = \epsilon(q,m,n)>0$ and $C= C(q,m,n,U)$ such that if $\|\Om\|_{\M^{2,n-2}(B_1)} \leq \epsilon$ then
\begin{equation*}
\|\D^2 u\|_{\M^{s,n-2}(U)} + \|\D u\|_{\M^{q,n-2}(U)} \leq C\|u\|_{L^1(B_1)}.
\end{equation*}
\end{cor}

\begin{remark}
In the case that $|\Om| \leq C|\D u|$ (as is the case for Harmonic maps) this automatically gives that when $\|\D u\|_{\M^{2,n-2}(B_1)}$ is small enough then $u \in W^{2,q}$ for some $q>n$ yielding $u \in C^{1,\gamma}$ for some $\gamma \in (0,1)$. If we knew furthermore that $\Om$ depended on $u$ and $\D u$ in a smooth way (as is  also the case for Harmonic maps) then we could immediately conclude smoothness by a simple bootstrapping argument using Schauder theory. Thus we recover a proof of the regularity of weakly stationary harmonic maps into Riemannian manifolds (away from a singular set $S$ with $H^{n-2}(S) = 0$). We also mention that passing to the smooth local estimates for harmonic maps also easily follows - we leave the details to the reader. 
\end{remark}

\begin{remark}
We remark that Theorem \ref{Main} and Corollary \ref{highint} should hold (with some added technicalities) given any smooth metric $g$ on $B_1$, with $u$, $\Om$ and $f$ as above weakly solving 
\begin{equation}\label{h}
-\Dl_g u = \langle \Om,\D u\rangle_g +f.
\end{equation} 
Following \cite{riviere_inventiones} or \cite{riviere_struwe} one can check that a harmonic map $u:(B_1,g)\to \Nn\emb \R^m$ solves \eqref{h} with $f\equiv 0$, and therefore the regularity theory for general domains  would follow from this. 
\end{remark}
The method we use to prove Theorem \ref{Main} is (for the most part) broadly the same as that employed in \cite{Sh_To}, the real difference comes in section \ref{pdecay1} where we obtain a decay type estimate \eqref{d1} using both Hardy and Morrey space methods via a slight improvement of a result of Adams \cite{adams_riesz}. 

To illustrate the requirement for an improved Adams result consider the case that $\Om$ is divergence free, i.e. that $\Om = \ast \ed \xi$ for some $\xi\in W_0^{1,2}(B_1,so(m)\otimes \wedge^{n-2}\R^n)$. Using H\"older's inequality and the results in \cite{clms} we have
$$-\Dl u = \ast (\ed \xi \wedge \ed u) \in \h^1\cap M^{1,n-2}$$
after extending $\xi$ by zero and extending $u$ to $W^{1,2}(\R^n)$ appropriately. This follows from an important result of Coiffman et al, \cite{clms} that (in particular) given two one forms $D,E \in L^2(\R^n,\bigwedge^1 \R^n)$ such that $\ed E =0$ and $\dv D = 0$ weakly, then $\langle E,D\rangle \in \h^1(\R^n)$ with $\|\langle E,D\rangle\|_{\h^1(\R^n)} \leq C\|E\|_{L^2(\R^n)}\|D\|_{L^2(\R^n)}.$
The key to attaining sub-critical integrability lies in finding better decay for the $L^2$ norm of $\D u$ on small balls, thus what we would like to do is estimate $\|\D u \|_{L^2}$. When $n=2$ we have that $\h^1\emb H^{-1}= (W_0^{1,2})^{\ast}$, thus getting such an estimate is not a problem. In higher dimensions one certainly does not have $\h^1\emb H^{-1}$, however by slightly improving a result of Adams \cite{adams_riesz} we are able to prove that $\h^1\cap M^{1,n-2}\emb H^{-1}$ and we can estimate $\|\D u\|_{L^2}$. Actually we prove a local version: $h^1\cap M^{1,n-2}\emb H^{-1}(K)$ - Corollary \ref{adams}. By re-writing with respect to a Coulomb gauge, and using the H\"older continuity already obtained in Theorem \ref{old}, we are able to essentially reduce to the above situation. Indeed, the anti-symmetry of $\Om$ can be thought of as replacing the condition that it is divergence free.  

Before we state the result we first introduce some notation. Let $N[g]$ be the Newtonian potential of a function $g$, that is convolution by the fundamental solution of the Laplacian $\Gamma:\R^n\sm\{0\}\to \R$. We have $\Gamma(x) = C(n)|x|^{2-n}$ for $n\geq 3$, where $C(n)$ is a dimensional constant, and $\Gamma(x) = -\fr1{2\pi}\log|x|$ when $n=2$. Thus $N[g] = \Gamma \ast g$ and $-\Dl N[g] = g$. We will need to estimate the operation $g\mapsto \D N[g]$ which is a convolution operator given by $\D \Gamma \ast g$.

For $0<\al<n$, define $A_{\al}$ to be convolution of a function by $a_{\al}$ where $a_{\al}$ is homogeneous of degree $\al -n$ and is smooth as a function on the sphere $S^{n-1}$ with $\|a_{\al}\|_{C^1(S^{n-1})}\leq C(n)$. For instance $N$ is of the form $A_2$ when $n\geq 3$, moreover the operator defined by taking a derivative of the Newtonian potential $\D N[g]$ is of the form $A_{1}$ for all $n\geq 2$. The $A_{\al}$ are essentially Riesz potentials. 

A few words of warning are necessary here: our notation differs from that used in \cite{adams_riesz}, in particular we define our maximal functions $M_{\beta}$ differently and the Riesz potentials we consider are slightly more general. 

The  next theorem is a replacement of a weak $L^q$-estimate given by Proposition 3.2 in \cite{adams_riesz}. 
We replace the maximal function $M[g]$ by $g_{\ast}$ thereby allowing us to estimate functions in the (local) Hardy space. 
\begin{theorem}\label{Adams}
Let $0\leq \beta <n$, $0<\al p < n-\beta$ and $g\in \M^{p,\beta}(\R^n)$ we have 
$$|A_{\alpha}[g](x)|\leq C(\al, \beta, n,p) (M_{\beta}[g^p](x))^{\fr{\al}{n-\beta}}(g_{\ast}(x))^{\fr{n-\beta-\al p}{n-\beta}},$$
with $C(\al, \beta, n, p) \leq   C(n) \sup\left\{\fr{1}{1-\left(\fr12\right)^{\al}},\fr{1}{1-\left(\fr12\right)^{\al-\fr{n-\beta}{p}}} \right\}.$
\end{theorem}
As alluded to above we have the following
\begin{cor}\label{adams}
Let $g\in h^1\cap M^{1,n-2}(\R^n)$, then for any compact $K\In \R^n$ there exists some $C=C(K,n)$ such that 
$$\|g\|_{H^{-1}(K)}\leq C \|g\|_{\M^{1,n-2}}^{\fr12}\|g\|_{h^1}^{\fr12}.$$
\end{cor}

Theorem \ref{Adams} also recovers the following potential estimates, due to Adams \cite{adams_riesz}. We provide a proof of the following because we wish to keep track of the constants more closely.
\begin{cor}
\label{adams1}
Let $0\leq \beta < n$, $0<\al p<n-\beta$ and $\fr1{\tilde{p}} = \fr1p - \fr{\al}{n-\beta}$.
\begin{enumerate}\item When $p>1$ we have that 
$A_{\al}:\M^{p,\beta}(\R^n) \tu \M^{\ti{p},\beta}(\R^n)$ 
is bounded. Moreover there exists 
$$C(n,p,\al,\beta)\leq C(n) \left(\fr{p}{p-1}\right)^{\fr{1}{\ti{p}}}\sup\left\{\fr{1}{1-\left(\fr12\right)^{\al}},\fr{1}{1-\left(\fr12\right)^{\al-\fr{n-\beta}{p}}} \right\}$$ such that $\|A_{\al}[g]\|_{\M^{\ti{p},\beta}(\R^n)} \leq C(n,p,\al,\beta) \|g\|_{\M^{p,\beta}(\R^n)}.$
\item When $p=1$ we have that $A_{\al}:\M^{1,\beta}(\R^n) \tu \M^{(\ti{p},\infty),\beta}(\R^n)$ is bounded.
Moreover there exists 
$$C(n,\al,\beta)\leq C(n) \sup\left\{\fr{1}{1-\left(\fr12\right)^{\al}},\fr{1}{1-\left(\fr12\right)^{\al-(n-\beta)}} \right\}$$ such that $\|A_{\al}[g]\|_{\M^{(\ti{p},\infty),\beta}(\R^n)} \leq C(n,\al,\beta) \|g\|_{\M^{p,\beta}(\R^n)}.$

\end{enumerate}
\end{cor}
These reduce to well known estimates when $\beta =0$.
\begin{remark}\label{const}
Setting $\al =1$ and recalling that $\D N[g]$ is an operator of the form $A_{1}$ we see that given $a,b$ with $1<a\leq p \leq b < n-\beta$ then there is a uniform constant $C=C(a,b)$ such that for $g \in \M^{p,\beta}$ then $\|\D N[g]\|_{\M^{\ti{p},\beta}}\leq C\|g\|_{\M^{p,\beta}}$ for any $p$ and $\beta$ in this range.

Furthermore, let $u$, $\Om$ and $f$ be as in \eqref{ff}. Define $g:=\Om.\D u$ and extend both $f$ and $g$ by zero to $\R^n$. Then there is some harmonic function $h:B_1\to \R^m$ with $u = N[g] + N[f] + h$, since $-\Dl (u - N[g] - N[f]) = 0$ on $B_1$. Standard estimates (using Corollary \ref{adams1}) give 
$$\|h\|_{L^1(B_1)} \leq C(\|u\|_{L^1(B_1)} + \|f\|_{L^p(B_1)} + \|\Om\|_{M^{2,n-2}(B_1)}\|\D u\|_{M^{2,n-2}(B_1)})$$ and therefore since $h$ is harmonic, for any $r<1$ 
$$\|\D h\|_{L^{\infty}(B_r)} \leq \fr{C(n)}{(1-r)^{n+1}}(\|u\|_{L^1(B_1)} + \|f\|_{L^p(B_1)} + \|\Om\|_{M^{2,n-2}(B_1)}\|\D u\|_{M^{2,n-2}(B_1)}). $$ Thus we gain local control on $\D u = \D N[g] + \D N[f] + \D h$ by estimating $\D N[g]$ and $\D N[f]$.\end{remark}

\begin{remark}
We point out that by combining Theorem \ref{Adams} and Corollary \ref{adams1} we have $A_{\al}:\h^1\cap M^{1,\beta}(\R^n)\to L^{\ti{p}}\cap M^{(\ti{p},\infty),\beta}(\R^n)$ with an estimate - the details are left to the reader. In particular, when $\beta = 0$ we have $A_{\al}:\h^1(\R^n)\to L^{\fr{n}{n-\al}}(\R^n)$ - see also \cite[Theorem 1.77]{semmes_primer}.
\end{remark}

\emph{Acknowledgments:} I was supported by The Leverhulme Trust during the completion of this work.  I would also like to thank the referee for providing invaluable advice on the presentation of this paper. 

\section{Proof of Theorem \ref{Main}}\label{pMain}

We prove Theorem \ref{Main} based on the following proposition, analogous to \cite[Lemma 7.3]{Sh_To}, the proof of which is left to section \ref{pmain}.
\begin{prop}\label{main}
Let $u$, $\Om$ and $f$ be as in Theorem \ref{Main}. There exits $\epsilon = \epsilon (n,m,p)$ such that whenever $\|\Om\|_{\M^{2,n-2}(B_1)}^2 \leq \epsilon$ then $\D u \in \M_{loc}^{2,n-2(\fr np -1)}(B_1,\R^m)$.

\end{prop}
\begin{proof}[Proof of Theorem \ref{Main}]
This proof generalises the ideas needed in the proofs of \cite[Lemmata 7.1 and 7.2]{Sh_To} to Morrey spaces. 

Using the improved regularity from Proposition \ref{main} and the H\"older inequality we see that $\Om.\D u \in \M_{loc}^{1,n(1-\fr1p)}$. We also have $f\in M^{1,n(1-\fr1p)}$ - see appendix \ref{scal} if necessary. 

By applying Corollary \ref{adams1}  (for $\al =1$, $p =1$ and $\beta = n(1-\fr1p)$) and Remark \ref{const} we see that this implies $\D u \in \M_{loc}^{(\fr{n}{n-p},\infty),n(1-\fr1p)}$, which in turn gives, $\D u \in \M_{loc}^{\fr{\ta n}{n-p}, n(1-\fr{\ta}{p})}$ for any $\ta < 1$ - we leave the details to the reader.  

The fact that $\D u \in \M_{loc}^{\fr{\ta n}{n-p}, n(1-\fr{\ta}{p})}$ implies  $\Om.\D u \in \M_{loc}^{s,n(1-\fr{s}{p})}$ where $\fr1s=\fr12 + \fr{n-p}{\ta n}$. We can choose $\ta$ so that $s>1$ but note that we also have $s < \fr{2n}{3n-2p} < \fr{2p}{n}$ for $p \in (\fr{n}{2},n)$. 
We make the following claim: 
\begin{eqnarray*}
&\Om.\D u \in \M_{loc}^{s_k,n(1-\fr{s_k}{p})}, s_k \in (1,\fr{2p}{n})  \Rightarrow&  \\
&\Om. \D u \in \M_{loc}^{s_{k+1}, n(1-\fr{s_{k+1}}{p})}, s_k<s_{k+1}= \fr{2ns_k}{ns_k + 2(n-p)} \in (1,\fr{2p}{n}).&
 \end{eqnarray*}
To prove the claim we can easily check that $f \in \M^{s_k,n(1-\fr{s_k}{p})} $ with a uniform estimate for any such $s_k$ (see appendix \ref{scal}). Therefore we may apply Corollary \ref{adams1} (for $\al =1$, $p = s_k$ and $\beta = n(1-\fr{s_k}p)$) and Remark \ref{const} to yield $\D u \in \M_{loc}^{s_k\fr{n}{n-p},n(1-\fr{s_k}{p})}$. Again by H\"older's inequality we have $\Om.\D u \in \M_{loc}^{s_{k+1},n(1-\fr{s_{k+1}}{p})}$ where $\fr{1}{s_{k+1}}= \fr12 + \fr{n-p}{s_k n}$. 
We check that $$\fr{s_k}{s_{k+1}} = \fr{s_k}{2} + \fr{n-p}{n} < \fr{p}{n} + 1 - \fr{p}{n} =1.$$
If we assume, to get a contradiction, that $s_{k+1} \geq \fr{2p}{n}$ then we have 
$\fr{2ns_k}{ns_k + 2(n-p)} \geq \fr{2p}{n}.$
Which implies 
$2ns_k \geq 2ps_k +2(n-p)\fr{2p}{n}$
and therefore $s_k \geq \fr{2p}{n}$, a contradiction. Thus the claim holds. 

We have the recursive relation $s_{k+1}=\fr{2ns_k}{ns_k + 2(n-p)}$, so we have $s_k \uparrow \fr{2p}{n}$ and we have proved that $\Om.\D u \in \M_{loc}^{s,n(1-\fr{s}{p})}$ for all $s < \fr{2p}{n}$. Thus $\D u \in \M_{loc}^{s\fr{n}{n-p}, n(1-\fr{s}{p})}$ for all $s$ in this range (again by Corollary \ref{adams1}).

For some $B_R(x_0)\cemb B_1$ letting $\hat{u}(x) = u(x_0 + Rx)$, $\hat{\Om}(x) = R\Om(x_0+Rx)$ and $\hat{f}(x) = R^2 f(x_0 + Rx)$ we have 
$$-\Dl \hat{u} = \hat{\Om}.\D \hat{u} + \hat{f}$$ on $B_1$. The above argument gives us that $\D \hat{u} \in \M^{s\fr{n}{n-p}, n(1-\fr{s}{p})}(B_1)$ for all $s<\fr{2p}{n}$. 
 
Now let $s \in (\fr{2p+n}{2n},\fr{2p}{n})$ and denote by $t$ the next value given in the bootstrapping argument (if $s=s_k$ then $t=s_{k+1}$).  Notice that within this range, by Remark \ref{const} there is a $C=C(n,p)$ independent of $s$ and $t$ such that 
\begin{eqnarray*}
\|\D \hat{u}\|_{\M^{t\fr{n}{n-p},n(1-\fr{t}{p})}(B_{\fr12})} &\leq& C(\|\hat{\Om}\|_{\M^{2,n-2}(B_1)}\|\D \hat{u}\|_{\M^{s\fr{n}{n-p},n(1-\fr{s}{p})}(B_{1})} + \|\hat{f}\|_{L^p(B_1)} + \|\hat{u}\|_{L^1(B_1)}) \nonumber\\
&\leq& C(\|\hat{\Om}\|_{\M^{2,n-2}(B_1)}\|\D \hat{u}\|_{\M^{t\fr{n}{n-p},n(1-\fr{t}{p})}(B_{1})} +\|\hat{f}\|_{L^p(B_1)} + \|\hat{u}\|_{L^1(B_1)})\nonumber
\end{eqnarray*}
since whenever $1<s<t<\fr{2p}{n}$ we have the estimate 
$$\|\D \hat{u}\|_{\M^{s\fr{n}{n-p},n(1-\fr{s}{p})}}\leq C\|\D \hat{u}\|_{\M^{t\fr{n}{n-p},n(1-\fr{t}{p})}}$$ for $C=C(n,p)$.
Raising to the power $\mu : = t\fr{n}{n-p}$ we see that 
\begin{equation*}
\|\D \hat{u}\|_{\M^{\mu,n(1-\fr{t}{p})}(B_{\fr12})}^{\mu} \leq C(\|\hat{\Om}\|_{\M^{2,n-2}(B_1)}^{\mu}\|\D \hat{u}\|_{\M^{\mu,n(1-\fr{t}{p})}(B_{1})}^{\mu} +(\|\hat{f}\|_{L^p(B_1)} + \|\hat{u}\|_{L^1(B_1)})^{\mu}),
\end{equation*}
where we can still pick $C$ independent of $t$ since $\mu < \fr{2p}{n-p}$. 
Undoing the scaling leaves (see appendix \ref{scal})
\begin{eqnarray*}
R^{\mu + \fr{nt}{p}}\|\D u\|_{\M^{\mu,n(1-\fr{t}{p})}(B_{\fr{R}{2}}(x_0))}^{\mu} &\leq& C(\|\Om\|_{\M^{2,n-2}(B_1)}^{\mu}R^{\mu + \fr{nt}{p}}\|\D u\|_{\M^{\mu,n(1-\fr{t}{p})}(B_{R}(x_0))}^{\mu} + \\ &&+(R^{2-\fr{n}{p}}\|f\|_{L^p(B_1)} + R^{-n}\|u\|_{L^1(B_1)})^{\mu}).
\end{eqnarray*}
Since $R <1$, $\mu < \fr{2p}{n-p}$ and $t<\fr{2p}{n}$ we have that 
\begin{eqnarray*}
\|\D u\|_{\M^{\mu,n(1-\fr{t}{p})}(B_{\fr{R}{2}}(x_0))}^{\mu} &\leq& C\|\Om\|_{\M^{2,n-2}(B_1)}^{\mu}\|\D u\|_{\M^{\mu,n(1-\fr{t}{p})}(B_{R}(x_0))}^{\mu}+\\ &&+C(\|f\|_{L^p(B_1)} + \|u\|_{L^1(B_1)})^{\mu})R^{-((n+1)\fr{2p}{n-p} + 2)}.
\end{eqnarray*} 
We are now in a position to apply Lemma \ref{LS} for $\Gamma = C(\|f\|_{L^p(B_1)} + \|u\|_{L^1(B_1)})^{\mu})$, $\epsilon \leq \left(\fr{\epsilon_0}{C}\right)^{\fr{n-p}{2p}}$ and $\epsilon_0=\epsilon_0(n,k)$ is found for $k= (n+1)\fr{2p}{n-p} + 2$ to give the estimate 
\begin{equation*}
\|\D u\|_{\M^{\mu,n(1-\fr{t}{p})}(B_{\fr{1}{2}})} \leq C(\|f\|_{L^p(B_1)} + \|u\|_{L^1(B_1)})
\end{equation*}
with $C$ independent of $t$. We may now pass to the limit $t \uparrow \fr{2p}{n}$ to give
\begin{equation*}
\|\D u\|_{\M^{\fr{2p}{n-p},n-2}(B_{\fr{1}{2}})} \leq C(\|f\|_{L^p(B_1)} + \|u\|_{L^1(B_1)}).
\end{equation*}
For the second estimate, note that we have $\Om.\D u +f \in \M_{loc}^{\fr{2p}{n},n-2}$ by H\"older's inequality, thus by Theorem \ref{peetre} and the proceeding remarks we have finished the proof after applying a covering argument. 

\end{proof}

\section{Proof of Proposition \ref{main}}\label{pmain}

We begin with a proposition stating the main decay estimate required, the proof of this is left until section \ref{pdecay1}. This decay estimate is analogous to that of part 2. from \cite[Theorem 1.5]{Sh_To}, except that here we crucially require the H\"older regularity already obtained in order to prove \eqref{d1}.

\begin{prop}\label{decay1}
With the set-up as in Theorem \ref{Main}. Let $\dl>0$, then there exist $\epsilon=\epsilon(n,m,p)>0$ and $C=C(\dl,m,n)$ such that when $\|\Om\|_{\M^{2,n-2}(B_1)} \leq \epsilon$ we have the following estimate ($\gamma = 2-\fr np$)
\begin{equation}\label{d1}
\|\D u\|_{L^2(B_r)}^2 \leq C(\dl)(\|\Om\|_{\M^{2,n-2}(B_1)}^2[u]_{C^{0,\gamma}(B_1)}^2 +  \|f\|_{L^p(B_1)}^2) + r^n(1+\dl)\|\D u\|_{L^2(B_1)}^2.
\end{equation}

\end{prop}

\begin{proof}[Proof of Proposition \ref{main}]
We follow the argument for the proof of  \cite[Lemma 7.3]{Sh_To}: Pick $\dl= \dl(n,p)$ sufficiently small so that 
\begin{equation}\label{lambda}
\lambda:=\fr {1 + \dl}{2^n} < \fr 1{2^{n-2+2\gamma}} :=\Lambda \in \left(\fr{1}{2^n} , 1\right)
\end{equation}
since $\gamma = 2-\fr np \in (0,1)$. 
 
Consider the solution on some small ball $B_R(x_0) \In B_1$ by re-scaling via $\hat{u}$, see appendix \ref{scal}. Since the hypotheses of Proposition \ref{decay1} are also satisfied for $\hat{u}$ we have (by \eqref{d1}) 
$$\|\D \hat{u}\|_{L^2(B_r)}^2 \leq C(\|\Om\|_{\M^{2,n-2}(B_1)}^2[\hat{u}]_{C^{0,\gamma}(B_1)}^2 + \|\hat{f}\|_{L^p(B_1)}^2) + r^n(1+\dl)\|\D \hat{u}\|_{L^2(B_1)}^2, $$
and setting $r=\fr12$ yields
\begin{equation*}
\|\D \hat{u}\|_{L^2(B_{\fr12})}^2 \leq  \lambda \|\D \hat{u}\|_{L^2(B_1)}^2+ C(\|\Om\|_{\M^{2,n-2}(B_1)}^2[\hat{u}]_{C^{0,\gamma}(B_1)}^2 + \|\hat{f}\|_{L^p(B_1)}^2). 
\end{equation*}
Undoing the scaling gives 
\begin{equation*}
\|\D u\|_{L^2(B_{\fr{R}{2}}(x_0))}^2 \leq  \lambda \|\D u\|_{L^2(B_R)}^2+ CR^{n-2+2\gamma}(\|\Om\|_{\M^{2,n-2}(B_1)}^2[u]_{C^{0,\gamma}(B_R(x_0))}^2 + \|f\|_{L^p(B_R(x_0))}^2).
\end{equation*}
Therefore, setting $R=2^{-k}$, $k\in \mathbb{N}_0$ and $a_k:= \|\D u\|_{L^2(B_{2^{-k}})}^2$ we have 
\begin{eqnarray*}
a_{k+1} &\leq& \lambda a_{k} + \left(\fr12\right)^{k(n-2+2\alpha)}C(\|\Om\|_{\M^{2,n-2}(B_1)}^2[u]_{C^{0,\gamma}(B_1)}^2 + \|f\|_{L^p(B_1)}^2)\\
&=& \lambda a_{k} + \Lambda^kC(\|\Om\|_{\M^{2,n-2}(B_1)}^2[u]_{C^{0,\gamma}(B_1)}^2 + \|f\|_{L^p(B_1)}^2).
\end{eqnarray*}
This can be solved to yield (letting $K:=C(\|\Om\|_{\M^{2,n-2}(B_1)}^2[u]_{C^{0,\gamma}(B_1)}^2 + \|f\|_{L^p(B_1)}^2)$)
$$a_{k+1}\leq \lambda^k a_1
+K\Lambda\frac{(\Lambda^k-\lambda^k)}{\Lambda-\lambda},$$
and by \eqref{lambda}, this simplifies to
$$\|\nabla u \|_{L^2(B_{2^{-k}}(x_0))}^2=:
a_k\leq C\Lambda^k.$$
Therefore, for $r\in (0,1/2]$
$$\|\nabla u \|_{L^2(B_r(x_0))}^2\leq Cr^{n-2+2\gamma}.$$
We have $$ \|\D u\|_{\M^{2,n-2+2\gamma}(B_{\fr 12})}^2 \leq C,$$
and a covering argument concludes the proof.

\end{proof}

\section{Proof of Proposition \ref{decay1}}\label{pdecay1}

\begin{proof}[Proof of Proposition \ref{decay1}]

We will use the Coulomb gauge in order to re-write our equation. Set $\epsilon$ small enough so that we can apply Lemma \ref{coulomb}. We have (see appendix \ref{hodge} for the relevant  background on Sobolev forms)
\begin{equation*}
\dv (P^{-1}\ed u) = \langle\dv \eta , P^{-1}\ed u\rangle +P^{-1}f
\end{equation*}
and 
\begin{equation*}
\ed (P^{-1}\ed u) = (\ed P^{-1} \wedge \ed u).
\end{equation*}
We can also set $\epsilon$ small enough in order to apply Theorem \ref{old} so that $u \in C^{0,\gamma}$ where $\gamma = 2-\fr{n}{p}$. Now we wish to extend the quantities arising above in the appropriate way: First of all we may extend $\eta$ by zero. We also extend $P-\fr{1}{|B_1|}\int_{B_1}P$ to $\ti{P} \in W^{1,2}\bigcap L^{\infty}(\R^n)$ and finally $u$ to $\ti{u} \in C^{0,\gamma}(\R^n)$ where each has compact support in $B_2$ (we may assume $u \in C^{0,\gamma}(\overline{B}_1)$).   

Note that we have $\|\D \ti{P}\|_{L^2} \leq C\|\D P\|_{L^2(B_1)} \leq C\|\Om\|_{\M^{2,n-2}(B_1)}$ by Poincar\'e's inequality and $\D \ti{P} = \D P$ in $B_1$. We also have $\ti{u} \in C^{0,\gamma}(\R^n)$ with $\|\ti{u}\|_{C^{0,\gamma}} \leq C \|u\|_{C^{0,\gamma}}$ and (since we may assume $\int u = 0 $) we have $\|\ti{u}\|_{C^{0,\gamma}} \leq C [u]_{C^{0,\gamma}}$, moreover $\ti{u} = u$ in $B_1$. All the constants here come from standard extension operators and are independent of the function, see for instance \cite{gt}.   
 
Now we use Lemma \ref{hodged} in order to write $P^{-1}\ed u = \ed a + \dv b + h$ with $a, b, h$ as in the Lemma. 
Notice that we have $\Dl a = \langle\dv \eta , P^{-1}\ed u\rangle +P^{-1}f$ and $\Dl b = \ed P^{-1} \wedge \ed u$ weakly. We proceed to estimate $\D u \in L^2$ by estimating $\|\ed a\|_{L^2}$, $\|\dv b\|_{L^2}$ and using standard properties of harmonic functions in order to deal with $\|h\|_{L^2}$. 

We start with $\|\ed a\|_{L^2}$; notice that $$\langle\dv \eta , P^{-1} \ed u\rangle = \langle\dv \eta, \ed (P^{-1} u)\rangle - \langle\dv \eta , \ed P^{-1}\rangle u = I + II.$$
For $I$, pick $\phi \in C_c^{\infty}(B_1)$ and check (we use that $\eta$ has zero boundary values)
\begin{eqnarray*}
\int \ast \langle\dv \eta, \ed (P^{-1} u)\rangle\phi &=&  (\dv \eta , \ed (P^{-1}u)\phi)  \\
&=& (\dv \eta ,\ed (P^{-1}u\phi)) - (\dv \eta , (\ed \phi)P^{-1}u) \\
&=& - (\dv \eta , (\ed \phi)P^{-1}u ) \\
&\leq& \|\D \eta\|_{L^2(B_1)}\|\D \phi\|_{L^2(B_1)}\|P^{-1}u\|_{L^{\infty}(B_1)}  \\
&\leq& C \|\Om\|_{\M^{2,n-2}(B_1)}\|\D \phi\|_{L^2(B_1)} [u]_{C^{0,\gamma}(B_1)}.
\end{eqnarray*}
We have $I \in H^{-1}(B_1)$ with 
\begin{equation}\label{I}
\|I\|_{H^{-1}(B_1)} \leq C \|\Om\|_{\M^{2,n-2}(B_1)}[u]_{C^{0,\gamma}(B_1)}.
\end{equation}
For $II$ notice that $\langle\dv \eta , \ed \ti{P}^{-1}\rangle\ti{u} = \langle\dv \eta, \ed P^{-1}\rangle u$ in $B_1$. Moreover $\langle\dv \eta , \ed \ti{P}^{-1}\rangle \in \mathcal{H}^1(\R^n)$ with
\begin{equation*}
\|\langle\dv \eta , \ed \ti{P}^{-1}\rangle\|_{\mathcal{H}^1}\leq C \|\D \eta\|_{L^2(B_1)}\|\D \ti{P}\|_{L^2(B_1)} \leq C\|\Om\|_{\M^{2,n-2}(B_1)}^2
\end{equation*}
by the result of Coiffman et al, \cite{clms}. 

The space $\mathcal{H}^1$ is \emph{not} stable by multiplication of smooth functions since $f \in \mathcal{H}$ implies $\int f = 0$. However for the local Hardy space $h^1$, as long as the multiplier function is sufficiently regular, we have stability. For instance if $h \in h^1$ and $g \in C^{0,\gamma}$, then $gh \in h^1$ and
$$\|gh\|_{h^1} \leq C(\gamma)\|g\|_{C^{0,\gamma}}\|h\|_{h^1}.$$Therefore,  $ \langle\dv \eta , \ed \ti{P}^{-1}\rangle\ti{u} \in h^1(\R^n)$ with
\begin{equation*}
\|\langle\dv \eta , \ed \ti{P}^{-1}\rangle\ti{u}\|_{h^1(\R^n)} \leq C \|\Om\|_{\M^{2,n-2}(B_1)}^2[u]_{C^{0,\gamma}(B_1)}.
\end{equation*}
We also have $\|M_{n-2}( \langle\dv \eta , \ed \ti{P}^{-1}\rangle\ti{u})\|_{L^{\infty}} \leq C \|\Om\|_{\M^{2,n-2}(B_1)}^2[u]_{C^{0,\gamma}(B_1)}$ since for $R>0$
\begin{equation*}
R^{2-n}\int_{B_R(x_0)} \langle\dv \eta , \ed \ti{P}^{-1}\rangle\ti{u} = R^{2-n}\int_{B_R(x_0) \bigcap B_1} \langle\dv \eta , \ed \ti{P}^{-1}\rangle\ti{u} \leq C \|\Om\|_{\M^{2,n-2}(B_1)}^2[u]_{C^{0,\gamma}(B_1)}
\end{equation*}
(remember $\eta$ was extended by zero). 
Now, using Corollarly \ref{adams} we have $\langle\dv \eta , \ed \ti{P}^{-1}\rangle\ti{u} \in H^{-1}(B_1)$, moreover $\langle\dv \eta , \ed \ti{P}^{-1}\rangle\ti{u} = \langle\dv \eta, \ed P^{-1}\rangle u$ in $B_1$ so 
\begin{equation}\label{II}
\|II\|_{H^{-1}(B_1)} \leq C \|\Om\|_{\M^{2,n-2}(B_1)}^2[u]_{C^{0,\gamma}(B_1)}.
\end{equation}
Putting  \eqref{I} and \eqref{II} together yields $\langle \dv \eta , P^{-1} \ed u\rangle \in H^{-1}(B_1)$ with (assuming $\epsilon <1$)
\begin{equation*}
\|\langle\dv \eta , P^{-1} \ed u\rangle\|_{H^{-1}(B_1)} \leq C \|\Om\|_{\M^{2,n-2}(B_1)}[u]_{C^{0,\gamma}(B_1)}.
\end{equation*}
It is easy to check that $P^{-1}f \in H^{-1}(B_1)$ with $\|P^{-1}f\|_{H^{-1}(B_1)} \leq C\|f\|_{L^p(B_1)}$, overall this means that $a \in W_0^{1,2}(B_1)$ weakly solves 
\begin{equation*}
\Dl a = \langle\dv \eta , P^{-1} \ed u\rangle + P^{-1}f,
\end{equation*}
so we have 
\begin{equation}\label{a}
\|\D a\|_{L^2(B_1)} \leq C (\|\Om\|_{\M^{2,n-2}(B_1)}[u]_{C^{0,\gamma}(B_1)} + \|f\|_{L^p(B_1)}).
\end{equation}
Now we need to estimate $\|\dv b\|_{L^2(B_1)}$. We know that $b \in W_N^{1,2}(B_1,\bigwedge^2 \R^n)$ (see appendix \ref{hodge} for a definition) has $ \ed b = 0 $ and $\Dl b = (\ed P^{-1} \wedge \ed u)$. We have
\begin{equation*}
\|\dv b\|_{L^2(B_1)} = \sup_{E \in C^{\infty}(B_1,\bigwedge^1 \R^n)\,\, \|E\|_{L^2(B_1)}\leq 1}  (\dv b, E).
\end{equation*}
Using a smooth version of Lemma \ref{hodged} we can decompose each $E$ by $E= \ed e_1 + \dv e_2 + e_3$ where $e_1 \in C_0^{\infty}(B_1)$, $e_2 \in C_N^{\infty}(B_1, \bigwedge^2 \R^n)$ with $\ed e_2 = 0$ and $\ed e_3 = \dv e_3 = 0$ ($e_3$ is a harmonic one form). Notice that $(\dv b, \ed e_1) = 0$ since $b$ has zero normal component and  $\ed^2 e_1 =0$. Also we have $(\dv b , e_3) = 0$ since $e_3$ is harmonic and $b$ has vanishing normal components. Therefore 
\begin{eqnarray*}
(\dv b, E) &=& (\dv b,  \dv e_2)\\
&=& ( P^{-1}  \ed u, \dv e_2)\\
&=& (\ed (P^{-1}u), \dv e_2) - ((\ed P^{-1})u,\dv e_2)  \\
&=& -((\ed P^{-1})u, \dv e_2)\\
&\leq& C \|\Om\|_{\M^{2,n-2}(B_1)}[u]_{C^{0,\gamma}(B_1)}\|\dv e_2\|_{L^2(B_1)}\\
&\leq& C\|\Om\|_{\M^{2,n-2}(B_1)}[u]_{C^{0,\gamma}(B_1)}\|E\|_{L^2(B_1)}.
\end{eqnarray*}
Therefore 
\begin{eqnarray}\label{b}
\|\dv b\|_{L^2(B_1)} &\leq& C\|\Om\|_{\M^{2,n-2}(B_1)}[u]_{C^{0,\gamma}(B_1)}.
\end{eqnarray}
We note here that by \cite[Theorem 7.5.1]{morrey} and $\ed b = 0$ that we in fact have the same estimate for $\D b$.

We now use the fact that $h$ is harmonic giving that the quantity $r^{-n}\|h\|_{L^2(B_r)}^2 $ is increasing, and Lemma \ref{hodged} to give
\begin{eqnarray*}
\|h\|_{L^2(B_r)}^2 &\leq& r^n \|h\|_{L^2(B_1)}^2 \\
&\leq& r^n\|P^{-1}\ed u\|_{L^2(B_1)}^2 \\
&=& r^n \|\ed u\|_{L^2(B_1)}^2
\end{eqnarray*}
where the last line follows because $P$ is orthogonal. 

Going back to our original Hodge decomposition we see that (using Young's inequality, the orthogonality of $P$, \eqref{a} and \eqref{b})
\begin{eqnarray*}
\|\ed u \|_{L^2(B_r)}^2 &=& \|P^{-1}\ed u \|_{L^2(B_r)}^2 \\
&\leq& (\|h\|_{L^2(B_r)} + \|\ed a\|_{L^2(B_r)} + \|\dv b\|_{L^2(B_r)})^2 \\
&\leq& (1+\dl)\|h\|_{L^2(B_r)}^2 + C_{\dl} (\|\ed a\|_{L^2(B_r)} + \|\dv b\|_{L^2(B_r)})^2 \\
&\leq& (1+\dl)r^n\|\ed u\|_{L^2(B_1)}^2 + C_{\dl}(\|\Om\|_{\M^{2,n-2}(B_1)}^2 [u]_{C^{0,\gamma}(B_1)}^2 + \|f\|_{L^p(B_1)}^2).
\end{eqnarray*}
\end{proof}

\section{Proof of Theorem \ref{Adams} and Corollaries \ref{adams}, \ref{adams1}}

\begin{proof}[Proof of Theorem \ref{Adams}]

First assume the theorem holds for $p=1$. If $q>1$ and $g\in \M^{q,\beta}$ then by H\"older's inequality we have $g\in \M^{1,\hat{\beta}}$ with
$$M_{\hat{\beta}}[g]\leq C(n)M_{\beta}[g^q]^{\fr1q}$$
where $n-\hat{\beta} = \fr1q(n-\beta)$. Thus if $\al q < n-\beta$ (giving $\al < n-\hat{\beta}$), by applying the theorem for $p=1$ we have 
\begin{eqnarray*}
|A_{\alpha}[g](x)| &\leq& C(\al, \hat{\beta}, n) (M_{\hat{\beta}}[g](x))^{\fr{\alpha}{n-\hat{\beta}}}(g_{\ast}(x))^{\fr {n-\hat{\beta} -\alpha}{n-\hat{\beta}}}\\
&\leq& C(\al, \beta, n, q) (M_{\beta}[g^q](x))^{\fr{\al}{n-\beta}}(g_{\ast}(x))^{\fr{n-\beta-\al q}{n-\beta}}.
\end{eqnarray*}
Where we know that
$$C(\al, \beta, n, q) \leq   C(n) \sup\left\{\fr{1}{1-\left(\fr12\right)^{\al}},\fr{1}{1-\left(\fr12\right)^{\al-\fr{n-\beta}{q}}} \right\}.$$ 
It remains to prove Theorem \ref{Adams} for $p=1$. 

We split $A_{\alpha}(g)$ up into its near and far parts using a partition of unity subordinate to dyadic annuli of a chosen modulus $\dl$. More precisely, let $\theta (x) \in C_c^{\infty}(B_4 \sm B_{\fr 12})$ with $\theta (x) >0 $ for $1\leq |x|\leq 2$. Similarly as is done in Semmes \cite[Theorem 1.77]{semmes_primer} we can arrange so that
$$\sum_{j \in \Z}\theta(\dl^{-1}2^{-j}x) = 1 $$ for all $x \in \R^n \sm \{0\}$. Moreover we want for our choice of $a_{\al}$ that $C\|\D[\theta(4\cdot)a_{\al}(\cdot)]\|_{L^{\infty}(B_1)} \leq 1$ for some constant $C(n)$ (for reasons that will become apparent below). We can always find such a $C$ since we have assumed a uniform $C^1$ bound on $a_{\al}$ when restricted to the sphere. Notice that $\theta(4\cdot) a(\cdot) \in C_c^{\infty}(B_1)$ also. 

Now define $\eta^j (x) : = \dl^{-1}2^{-j} \theta (\dl^{-1}2^{-j}x) a_{\al}(x)$. Notice that $\dl 2^j \eta^j(x)$ is the piece of $a_{\al}$ around $\dl 2^{j-1} \leq |x| \leq \dl 2^{j+2}$, so that 
$$A_{\alpha}(g) = \sum_{j \in \Z} \dl 2^j \eta^j \ast g = \sum_{j \leq 0}\dl 2^j \eta^j \ast g + \sum_{j \geq 1}\dl 2^j \eta^j \ast g = I_{inner} + I_{outer}. $$
The intuition here is that we use the decay estimate we have on $g$ in order to deal with $I_{outer}$ and we use the Hardy space qualities of $g$ in order to deal with $I_{inner}$. 

With that in mind we start with estimating $I_{inner}$. We use the following claim: $$|\eta^j \ast g(x)| \leq C(\dl^{-1}2^{-j})^{1-\alpha} g_{\ast}(x).$$ 
This is easy enough to see, first of all we remark that in our definition of $g_{\ast}$ we choose to use the function $\psi (x) := C \theta(4x) a_{\al}(x)$, therefore $g_{\ast}(x):= \sup_{t>0} |\psi_t \ast g (x)|$. 
\begin{eqnarray*}
|\eta^j \ast g(x)| &=& \left| \int_{B_{\dl 2^{j+2}}} \dl^{-1}2^{-j} \theta(\dl^{-1}2^{-j}(x-y))a_{\al}(x-y)g(y) dy \right| \\
&=& \left| \int_{B_{\dl2^{j+2}}} \dl^{-1}2^{-j}\theta(\dl^{-1}2^{-j}(x-y))a_{\al}(\dl^{-1}2^{-(j+2)}(x-y))g(y)(\dl^{-1}2^{-(j+2)})^{n-\alpha} dy \right| \\
&=& C (\dl^{-1}2^{-j})^{1-\alpha}|\psi_{\dl 2^{j+2}} \ast g (x)| \leq C (\dl^{-1}2^{-j})^{1-\alpha}g_{\ast}(x).
\end{eqnarray*}
We estimate
\begin{eqnarray*}
|I_{inner}| &\leq& \sum_{j \leq 0}  \dl 2^j |\eta^j \ast g (x)| \\
&\leq& C \sum_{j \leq 0}  \dl 2^j(\dl^{-1}2^{-j})^{1-\alpha}    g_{\ast}(x) \\
&\leq& C\fr{1}{1-\left(\fr12\right)^{\al}}\dl^{\alpha} g_{\ast}(x).
\end{eqnarray*} 
Now we estimate $I_{outer}$
\begin{eqnarray*}
|I_{outer}| &\leq&	C\sum_{j \geq 1} \int_{\dl 2^{j-1} \leq |x-y| \leq \dl 2^{j+2}} |\theta(\dl^{-1}2^{-j}(x-y))||a_{\al}(x-y)||g(y)|dy \\
&\leq& C\sum_{j \geq 1} \int_{\dl 2^{j-1} \leq |x-y| \leq \dl 2^{j+2}} |x-y|^{\alpha-n}|g(y)|dy\\
&\leq& C\sum_{j \geq 1} (\dl 2^{j-1})^{\alpha-n}\int_{|x-y| \leq \dl 2^{j+2}} |g(y)|dy\\
&\leq& C \sum_{j \geq 1} (\dl 2^{j-1})^{\alpha-n} (\dl 2^{j+2})^{\beta} M_{\beta}(g)(x)\\
&\leq& C \fr{1}{1-\left(\fr12\right)^{\al-(n-\beta)}}\dl^{\alpha -(n- \beta)} M_{\beta} (g)(x).
\end{eqnarray*}
Putting together these threads gives us $$|A_{\alpha}(g)| \leq C(\al,\beta,n)( \dl^{\alpha} g_{\ast}(x) + \dl^{\alpha - (n-\beta)}M_{\beta}(g)(x)).$$
Setting $\dl = \left(\fr {M_{\beta}(g)(x)}{g_{\ast}(x)}\right)^{\fr{1}{n-\beta}}$ gives
$$|A_{\alpha}(g)| \leq C(\al,\beta,n) (M_{\beta}(g)(x))^{\fr{\alpha}{n-\beta}}(g_{\ast}(x))^{\fr {n-\beta -\alpha}{n-\beta}}.$$ 
\end{proof}

\begin{proof}[Proof of Corollary \ref{adams}] 
By \cite[Theorem 1.92]{semmes_primer} we know that $g \in h^1$ if and only if for any $\psi \in C_c^{\infty}$ with $\int \psi \neq 0$ there is a constant $\lambda$ such that 
$\psi(g-\lambda) \in \mathcal{H}^1(\R^n)$,
with 
$$\|\psi(g-\lambda)\|_{\mathcal{H}^1(\R^n)} \leq C \|g\|_{h^1(\R^n)},$$
where $C = C(\psi)$ and $\lambda$ is chosen such that $\int \psi(g-\lambda) =0$ i.e. $\lambda:= \fr{\int g\psi}{\int \psi}$.

Therefore we let  $\ti{g} = \psi(g-\lambda) + \psi \lambda$ where $\psi \in C_c^{\infty}$, $\psi \equiv 1$ on $K$ and $0\leq \psi \leq 1$. 

We have $\ti{g} = g$ in $K$, $\psi(g-\lambda) \in \mathcal{H}^1$, and $$\|\lambda \psi\|_{L^{\infty}(K)}\leq C(K)\|g\|_{L^1(K)} \leq C(K)\|g\|^{\fr12}_{M^{1,n-2}}\|g\|^{\fr12}_{h^1}.$$ Clearly then we have 
$$\|A_1[\lambda \psi]\|_{L^2(K)} \leq C(n,K)\|g\|^{\fr12}_{M^{1,n-2}}\|g\|^{\fr12}_{h^1}.$$ Moreover by applying Theorem \ref{Adams}
\begin{eqnarray*}
\|A_1[\psi(g-\lambda)]\|_{L^2(K)} &\leq& C(n)\|\psi(g-\lambda)\|^{\fr12}_{M^{1,n-2}}\|\psi(g-\lambda)\|^{\fr12}_{\h^1} \\
&\leq& C(n,K) \|g\|^{\fr12}_{M^{1,n-2}}\|g\|^{\fr12}_{h^1}.
\end{eqnarray*}

Therefore $A_{1}[\ti{g}] \in L^2(K)$ with
$$\|A_{1}[\ti{g}]\|_{L^2(K)} \leq C\|g\|_{\M^{1,n-2}}^{\fr12}\|g\|_{h^1}^{\fr12}.$$
Now set $w = N[\ti{g}]= \Gamma \ast \ti{g}$ where $N$ is the Newtonian potential, we have that $\D_i w = \D_i N[\ti{g}] = \D_i \Gamma \ast \ti{g}$ is an operator of the form $A_{1}[g]$ for $a_1(x) = C(n)\fr {x_i}{|x|^n}$. Therefore for $\phi \in C_c^{\infty}(K)$ we can test 
$$\int_{K}g\phi = \int_{K} (\psi(g-\lambda) + \lambda \psi)\phi = -\int_{K} \Dl w \phi = \int_{K} \D w.\D \phi \leq C\|\D w\|_{L^2(K)}\|\phi\|_{W^{1,2}(K)}.$$
Thus $\|g\|_{H^{-1}(K)}\leq C \|g\|_{\M^{1,n-2}}^{\fr12}\|g\|_{h^1}^{\fr12}.$
\end{proof}

\begin{proof}[Proof of Corollary \ref{adams1}] 
We essentially follow \cite{adams_riesz} here, except that we keep track of the constants. 
Let $g\in \M^{p,\beta}$, $p\geq 1$ and $0<\al p < n-\beta$. We will show that 
$$\twopartdef{\|A_{\al}[g]\|_{L^{\ti{p}}(B_r(x))}^{\tilde{p}} \leq C\|g\|_{\M^{p,\beta}}^{\ti{p}}r^{\beta}}{p>1}{|\{z\in B_r(x) : |A_{\al}[g](z)|>s\}|\leq C\|g\|_{\M^{1,\beta}}^{\ti{p}}s^{-\ti{p}}r^{\beta}}{p=1.}$$
To that end write $g_r= g\chi_{B_{2r}(x)}$ and $g^r=g-g_r$. For $g_r$ we know that $\|g_r\|_{L^p} \leq Cr^{\fr{\beta}{p}}\|g\|_{\M^{p,\beta}}$. Thus by standard estimates for the maximal function (see \cite[page 5]{stein_singular} for example) we have 
$$\twopartdef{\|M[g_r]\|_{L^p} \leq C(n)\left(\fr{p}{p-1}\right)^{\fr1p}r^{\fr{\beta}{p}}\|g\|_{\M^{p,\beta}}}{p>1}{|\{x:M[g_r](x)>s\}|\leq C(n)s^{-1}r^{\beta}\|g\|_{\M^{1,\beta}}}{p=1.}$$
Now we can directly apply the estimate from Theorem \ref{Adams} to conclude (using the trivial estimate $g_{\ast}(x)\leq C(n)M[g](x)$)
$$\twopartdef{\|A_{\al}[g_r]\|_{L^{\ti{p}}(B_r(x))}^{\tilde{p}} \leq C(n,\al,\beta,p)^{\ti{p}}\fr{p}{p-1}\|g\|_{\M^{p,\beta}}^{\ti{p}}r^{\beta}}{p>1}{|\{z\in B_r(x) : |A_{\al}[g_r](z)|>s\}|\leq C(n,\al,\beta)^{\ti{p}}\|g\|_{\M^{1,\beta}}^{\ti{p}}s^{-\ti{p}}r^{\beta}}{p=1.}$$
Where $C(n,\al,\beta)$ is the constant appearing in Proposition \ref{Adams}. 

If $z\in B_r(x)$ then 
\begin{eqnarray*}
|A_{\al}[g^r](z)| &\leq& C\int_{\R^n\sm B_r(z)} \fr{1}{|z-y|^{n-\al}}|g(y)|\id y \nonumber \\
&\leq & C\sum_{j\geq 1} \int_{B_{2^{j}r}(z)\sm B_{2^{j-1}r}(z)} \fr{1}{|z-y|^{n-\al}}|g(y)|\id y \\
&\leq & C\sum (2^j r)^{\al -n}\|g\|_{\M^{p,\beta}}r^{n-\fr{n-\beta}{p}} \\
&=& C\fr{1}{1-\left(\fr12\right)^{\al - \fr{n-\beta}{p}}}\|g\|_{\M^{p,\beta}}r^{\al-\fr{n-\beta}{p}}\\
&\leq& C(n,\al, \beta , p)\|g\|_{\M^{p,\beta}}r^{\al-\fr{n-\beta}{p}}.
\end{eqnarray*}
Therefore for $p\geq 1$, $\|A_{\al}[g^r]\|_{L^{\ti{p}}(B_r(x))}^{\ti{p}} \leq C(n,\al, \beta,p)^{\ti{p}}\|g\|_{\M^{p,\beta}}^{\ti{p}} r^{\beta}.$
\end{proof}


\appendix

\section{Morrey, Campanato and H\"older spaces}\label{Morrey}
Here we state some important facts concerning Morrey spaces. Clearly $\M^{p,0} = L^p$ and $\M^{p,n} = L^{\infty}$, also we see that $M_n[\cdot]=M[\cdot]$ is the usual maximal function up to a constant. Also note that if we allow $\beta >n $ then $\M^{p, \beta} = \{0\}$. 

The related Campanato spaces $\Ca^{p,\beta}$ are variations on $BMO$, thus we try to capture an integral measure of oscillation similar to that for the Morrey spaces. For $g \in L^1(E)$ let $g_{r,x} = \dint_{B_r(x)\cap E} g $ and we say that $g \in \Ca^{p,\beta}(E)$ if $g \in L^p(E)$ and 
$$[g]_{\Ca^{p,\beta}(E)} := \sup_{x \in E,\,\,r>0} \left(r^{-\beta} \int_{B_r(x)\cap E} |g - g_{r,x}|^p \right)^{\fr 1p} <\infty$$
with norm (making $\Ca^{p,\beta}$ Banach spaces)
$$\|g\|_{\Ca^{p,\beta}(E)} = [g]_{\Ca^{p,\beta}(E)} + \|g\|_{L^p(E)}.$$
For Lipschitz domains we have $\M^{p,\beta} = \Ca^{p,\beta}$, when $0\leq \beta < n$ \cite[Chapter III, Proposition 1.2]{giaquinta}. However (modulo constants) $\Ca^{p,n} = BMO$ for all $p$ as opposed to $\M^{p,n} =L^{\infty}$.
We actually have that $\M^{p,\beta}\In \Ca^{p,\beta}$ with a uniform estimate (in $n$, $p$ and $\beta$). The reverse inclusion holds with an estimate whose constant blows up as $\beta$ approaches $n$.

Moreover $\Ca^{p,\beta}$ makes sense for $\beta >n $  and when $n < \beta \leq n+p$ we have $\Ca^{p,\beta} = C^{0,\gamma}$ with $\gamma = \fr{\beta - n}{p}$ \cite[Chapter III, Theorem 1.2]{giaquinta}. If $\beta > n+p $ then $\Ca^{p,\beta}$ are the constant functions. 

We say that $g \in \M_k^{p,\beta}$ if $g, \D^k g \in \M^{p,\beta}$. Using the Poincar\'e inequality we see that if $g \in \M_1^{p,\beta}$ for some $0\leq \beta \leq n$ then $g \in \Ca^{p,p+\beta}$. Therefore if $n-p<\beta \leq n$ then $g \in C^{0,\fr{\beta + p -n}{p}}$. Also the borderline case ($\beta = n-p$) gives $g \in BMO$. These last facts yield another proof of the Morrey embedding theorem: suppose $g\in W^{1,q}$ for $q>n$. Then $g\in \M_1^{n, n-\fr{n^2}{q}}\emb C^{0,1-\fr{n}{q}}$.

We also introduce here the related weak Morrey spaces $\M^{(p,\infty),\beta}(E)$, consisting of functions $g$ in the Lorentz space $L^{(p,\infty)}(E)$ or `weak $L^p$' with 
$$\|g\|_{\M^{(p,\infty),\beta}(E)}:=\sup_{x,\,\,r>0} r^{-\fr{\beta}{p}}\|g\|_{L^{(p,\infty)}(B_r(x)\cap E)} < \infty.$$ 
This condition is equivalent to 
$$|\{x\in B_r(x_0)\cap E : |g|(x) >s\}|\leq C\|g\|_{\M^{(p,\infty),\beta}(E)}^p s^{-p}r^{\beta}$$
with $C$ independent of $x_0$ and $r$.

Even though the Campanato spaces do not interpolate the $L^p$ spaces we still have good estimates on singular integrals. The following result of Peetre \cite{peetre_cam} generalises both Calderon-Zygmund and Schauder estimates. We consider the singular integral operator given by the operation $g\mapsto \D^2 N[g]$. In words it is the operator that maps a function to second order derivatives of its Newtonian potential. 
\begin{theorem}\label{peetre} For $1<p<\infty$ and $0\leq \beta < n+p$,  
$\D^2 N: \Ca^{p,\beta} \rightarrow \Ca^{p,\beta}$ is a bounded operator.
\end{theorem}
Therefore we also have for $1<p<\infty$ and $0\leq \beta < n$ that $\nabla^2 N : \M^{p,\beta} \rightarrow \M^{p,\beta}$ is bounded.

\section{Hodge Decompositions and Coulomb gauge}\label{hodge}

We require the following from \cite{riviere_struwe}, which shows that we can still find the appropriate Coulomb gauge even in the Morrey space setting. 

\begin{lemma}\label{coulomb}
Let $\Om \in \M^{2, n-2}(B_1, so(m) \otimes \wedge^1\R^n)$. Then there exists $\epsilon >0$ such that whenever $\|\Om\|_{\M^{2,n-2}(B_1)}^2 \leq \epsilon$ there exist $P \in W^{1,2}(B_1,SO(m))$ and $\eta \in W_0^{1,2}(B_1, so(m) \otimes \wedge^{2} \R^n)$ such that $\ed \eta =0$ on $B_1$ and
\begin{equation*}
P^{-1}\ed P + P^{-1}\Om P = \dv \eta.
\end{equation*}
Moreover $\D P, \D \eta \in \M^{2,n-2}(B_1)$ with 
\begin{equation*}
\|\D P\|^2_{\M^{2,n-2}(B_1)} + \|\D \eta\|^2_{\M^{2,n-2}(B_1)} \leq C\|\Om\|^2_{\M^{2,n-2}(B_1)} \leq C \epsilon. 
\end{equation*}
\end{lemma}
\begin{remark}
This appears different to the lemma appearing in \cite{riviere_struwe}, however upon replacing $\eta$ with $(-1)^{3n+1}\ast \xi$ as it appears in \cite{riviere_struwe} we see they are the same. The notation here should be explained, $\ed$ is the exterior derivative, $\dv = (-1)^{k(n-k)}\ast \ed \ast$ is the divergence operator on $k$-forms (formal adjoint of exterior derivative) and for any form $\om$, $\D \om$ refers to the collection of all first order derivatives, as opposed to $\ed \om $ or $\dv \om$ which refers to new forms comprised of certain combinations of first order derivatives of $\om$. Of course $\ast$ is the hodge star operator. 
\end{remark}

We recall here that there is a natural point-wise inner product for $k$-forms given by $\langle \om^1, \om^2\rangle =\ast( \om^1 \wedge \ast \om^2)$ and an $L^2$-inner product given by $(\om^1 , \om^2 ) = \int \ast \langle \om^1, \om^2\rangle$.  

Our main reference here is \cite[Chapter 7]{morrey} where we can find all of the results stated below, in particular we require the following.

\begin{lemma}\label{hodged}
Suppose $\om \in L^2(B_1,\bigwedge^1 \R^n)$ then there are unique $ a \in W_0^{1,2}(B_1)$, $b \in W_N^{1,2}(B_1,\bigwedge^2 \R^n)$ and a harmonic one form $h \in L^2(B_1,\bigwedge^1 \R^n)$ such that 
$$\om = \ed a + \dv b + h.$$
Moreover $\ed b = 0$ with
$$\|a\|_{W^{1,2}(B_1)} + \|b\|_{W^{1,2}(B_1)} + \|h\|_{L^2(B_1)} \leq C\|\om\|_{L^2(B_1)}$$
and
$$\|\ed a\|_{L^2(B_1)}^2 + \|\dv b\|_{L^2(B_1)}^2 + \|h\|_{L^2(B_1)}^2 = \|\om\|_{L^2(B_1)}^2.$$
\end{lemma}

We note here that $W_N^{1,2}(B_1,\bigwedge^k \R^n)$ is the space of forms whose normal boundary part vanishes, which we may define in a trace sense or equivalently for any smooth $k-1$ form $\nu$ we have $(\om, \ed \nu) = (\dv \om, \nu)$ when $\om \in W_N^{1,2}(B_1,\bigwedge^k \R^n)$. 

Otherwise we have the more general formula for smooth $k$ and $k-1$ forms respectively
$$(\om ,\ed \nu) = (\dv \om , \nu) + \int_{\partial} \nu_T \wedge \ast \om_N$$
where $_T$ and $_N$ denote the tangential and normal components. (The latter holds for any appropriate Sobolev forms by approximation). Note we could easily define $W_T^{1,2}(B_1,\bigwedge^k \R^n)$ in a weak sense also. 
We use the following fact: For $a \in W^{1,2}(B_1,\bigwedge^{k-1}\R^n)$, $b \in W^{1,2}(B_1,\bigwedge^k\R^n)$ we have $(\ed a, \dv b ) = 0$ if either $a_T = 0$ or $b_N = 0$. 

Note that we have $\Dl a = \dv \om$ and $\Dl b = \ed \om$ in a weak sense since $\ed h = \dv h = 0$, $\ed b = 0$ and since $a$ is a function ($\Dl = \ed \dv + \dv \ed$).

\section{Absorption lemma}

We have changed the hypotheses of the following lemma compared to how it appears in \cite{simon_regularity} and \cite{Sh_To}, however upon inspection of the proof (which can be found in \cite{simon_regularity}) it can be checked that the lemma as it is stated here is also proved.
 
\begin{lemma} (Leon Simon 
\cite[\S 2.8, Lemma 2]{simon_regularity}.)
\label{LS}
Let $B_{\rho}(y) \subset \R^n$ be any ball, $k \in \R$, $\Gamma >0$, and let $\varphi$ be any $[0,\infty)$-valued convex sub additive function on the collection of balls in $B_{\rho}(y)$; thus $\varphi(A) \leq \sum_{j=1}^{N} \varphi(A_j)$ whenever $A,A_1,A_2,....,A_N$ are balls in $B_{\rho}(y)$ with $A\subset \bigcup_{j=1}^{N} A_j$ and $A \bigcap A_j \neq \emptyset$ for any $j$. There is an $\epsilon_0 = \epsilon_0 (k,n)>0$ such that if 
\begin{equation*}\label{leon}
\sigma^k \varphi(B_{\sigma /2}(z)) \leq \epsilon_0 \sigma^k \varphi(B_{\sigma}(z)) + \Gamma
\end{equation*}  
whenever $B_{2\sigma}(z) \subset B_{\rho}(y)$, then there exists some $C=C(k,n)<\infty$ such that
\begin{equation*}
\rho^k \varphi(B_{\rho /2}(y)) \leq C\Gamma.
\end{equation*} 
\end{lemma}

In particular we can apply this lemma when $\varphi(A)= \|k\|_{\M^{p,\beta}(A)}^p$.

\section{Scaling}\label{scal}
We will need to consider $u$, $\Om$ and $f$ solving $$-\Dl u = \Om .\D u + f$$ on some small ball $B_R(x_0) \In B_1$. In order to do so we re-scale $\hat{u}(x):= u(x_0+Rx)$, $\hat{\Om}(x):=R\Om(x_0+Rx)$ and $\hat{f}:=R^2f(x_0+Rx)$. First of all we see that $$-\Dl \hat{u} = \hat{\Om}.\D \hat{u} + \hat{f}$$ on $B_1$ and we list the scaling properties of the related norms as follows.
\begin{itemize}
\item
$\|\hat{\Om}\|_{\M^{2,n-2}(B_1)} = \|\Om\|_{\M^{2,n-2}(B_R(x_0))}$. 
\item
$[\hat{u}]_{C^{0,\gamma}(B_1)} =  R^{\gamma}[u]_{C^{0,\gamma}(B_R(x_0))}$.
\item
$\|\hat{u}\|_{L^1(B_1)} = R^{-n} \|u\|_{L^1(B_R(x_0))}$.
\item
$\|\D \hat{u}\|_{\M^{l,\nu}(B_1)} = R^{\fr{l-(n-\nu)}{l}}\|\D u\|_{\M^{l,\nu}(B_R(x_0))}$ and setting $\nu =0$ gives $\|\D \hat{u}\|_{L^l(B_1)} = R^{1-\fr{n}{l}}\|\D u\|_{L^l(B_R(x_0))}$.

\item
We also have that the Lorentz spaces $L^{(l,\infty)}$ or `weak'-$L^l$ scale in the same fashion as the usual $L^l$ spaces,
$\|\D \hat{u}\|_{L^{(l,\infty)}(B_1)} = R^{1-\fr{n}{l}}\|\D u\|_{L^{(l,\infty)}(B_R(x_0))}.$
\item
$\|\hat{f}\|_{L^p(B_1)} = R^{2-\fr{n}{p}}\|f\|_{L^p(B_R(x_0))}$.
\item
For $f \in L^p(B_1)$ and $1\leq s\leq p$ we have 
$\|f\|_{\M^{s,n(1-\fr{s}{p})}(B_1)} \leq C\|f\|_{L^p(B_1)}$
for $C=C(n,p)$.\end{itemize}
\bibliographystyle{alpha}

\begin{thebibliography}{CLMS93}

\bibitem[Ada75]{adams_riesz}
David~R. Adams.
\newblock A note on {R}iesz potentials.
\newblock {\em Duke Math. J.}, 42(4):765--778, 1975.

\bibitem[Bet93]{bethuel_singset}
Fabrice Bethuel.
\newblock On the singular set of stationary harmonic maps.
\newblock {\em Manuscripta Math.}, 78(4):417--443, 1993.

\bibitem[CJWZ13]{CJWZ}  Q. Chen, J. Jost, G. Wang, M. Zhu. The boundary value problem for Dirac-harmonic maps.
 \emph{J. Eur. Math. Soc. (JEMS)}, Volume 15, Issue 3, 2013, 997--1031.

\bibitem[CLMS93]{clms}
R.~Coifman, P.-L. Lions, Y.~Meyer, and S.~Semmes.
\newblock Compensated compactness and {H}ardy spaces.
\newblock {\em J. Math. Pures Appl. (9)}, 72(3):247--286, 1993.

\bibitem[Eva91]{evans_spheres}
Lawrence~C. Evans.
\newblock Partial regularity for stationary harmonic maps into spheres.
\newblock {\em Arch. Rational Mech. Anal.}, 116(2):101--113, 1991.

\bibitem[Gia83]{giaquinta}
Mariano Giaquinta.
\newblock {\em Multiple integrals in the calculus of variations and nonlinear
  elliptic systems}, volume 105 of {\em Annals of Mathematics Studies}.
\newblock Princeton University Press, Princeton, NJ, 1983.

\bibitem[GT01]{gt}
David Gilbarg and Neil~S. Trudinger.
\newblock {\em Elliptic partial differential equations of second order}.
\newblock Classics in Mathematics. Springer-Verlag, Berlin, 2001.
\newblock Reprint of the 1998 edition.

\bibitem[H{\'e}l91]{helein_regularity}
Fr{\'e}d{\'e}ric H{\'e}lein.
\newblock R\'egularit\'e des applications faiblement harmoniques entre une
  surface et une vari\'et\'e riemannienne.
\newblock {\em C. R. Acad. Sci. Paris S\'er. I Math.}, 312(8):591--596, 1991.

\bibitem[Kel10]{keller}
Laura Gioia~Andrea Keller.
\newblock {$L^{\infty}$} estimates and integrability by compensation in
  Besov-Morrey spaces and applications.
\newblock {\em Adv. Calc. Var}, 5(3):285--327, 2012. 


\bibitem[LZ09]{LZ}
J.~Li and X.~Zhu.
\newblock Small energy compactness for approximate harmonic mappings.
\newblock {\em  Commun. Contemp. Math}, 13(5):741--763, 2011.

\bibitem[Mor66]{morrey}
Charles~B. Morrey, Jr.
\newblock {\em Multiple integrals in the calculus of variations}.
\newblock Die Grundlehren der mathematischen Wissenschaften, Band 130.
  Springer-Verlag New York, Inc., New York, 1966.
  
\bibitem[Mos12]{moser_regularity}
R. Moser
\newblock An ${L}^p$ regularity theory for harmonic maps. 
\newblock http://opus.bath.ac.uk/29504/
\newblock Preprint, 2012.

\bibitem[MS09]{schikorra_boundary}
Frank M{\"u}ller and Armin Schikorra.
\newblock Boundary regularity via {U}hlenbeck-{R}ivi\`ere decomposition.
\newblock {\em Analysis (Munich)}, 29(2):199--220, 2009.

\bibitem[Pee66]{peetre_cam}
Jaak Peetre.
\newblock On convolution operators leaving {$L^{p,}\,^{\lambda }$} spaces
  invariant.
\newblock {\em Ann. Mat. Pura Appl. (4)}, 72:295--304, 1966.

\bibitem[Riv07]{riviere_inventiones}
Tristan Rivi{\`e}re.
\newblock Conservation laws for conformally invariant variational problems.
\newblock {\em Invent. Math.}, 168(1):1--22, 2007.

\bibitem[Riv92]{riviere_singular}
Tristan Rivi{\`e}re.
\newblock Applications harmoniques de {$B^3$} dans {$S^2$} partout
  discontinues.
\newblock {\em C. R. Acad. Sci. Paris S\'er. I Math.}, 314(10):719--723, 1992.

\bibitem[RL13]{riviere_laurain}
Tristan Rivi\`ere and Paul Laurain.
\newblock Angular energy quantization for linear elliptic systems with
  antisymmetric potentials and applications.
\newblock To appear {\em Anal. PDE}, 2013.

\bibitem[RS08]{riviere_struwe}
Tristan Rivi{\`e}re and Michael Struwe.
\newblock Partial regularity for harmonic maps and related problems.
\newblock {\em Comm. Pure Appl. Math.}, 61(4):451--463, 2008.

\bibitem[Rup08]{rupflin}
Melanie Rupflin.
\newblock An improved uniqueness result for the harmonic map flow in two
  dimensions.
\newblock {\em Calc. Var. Partial Differential Equations}, 33(3):329--341,
  2008.

\bibitem[Sch10]{schikorra_frames}
Armin Schikorra.
\newblock A remark on gauge transformations and the moving frame method.
\newblock {\em Ann. Inst. H. Poincar\'e Anal. Non Lin\'eaire}, 27(2):503--515,
  2010.
  
\bibitem[Sch12]{schikorra_higher}
A. Schikorra.
\newblock $\epsilon$-regularity for systems involving non-local, antisymmetric operators.
\newblock arXiv:1205.2852
\newblock Preprint, 2012.

\bibitem[Sem94]{semmes_primer}
Stephen Semmes.
\newblock A primer on {H}ardy spaces, and some remarks on a theorem of {E}vans
  and {M}\"uller.
\newblock {\em Comm. Partial Differential Equations}, 19(1-2):277--319, 1994.

\bibitem[Sim96]{simon_regularity}
Leon Simon.
\newblock {\em Theorems on regularity and singularity of energy minimizing
  maps}.
\newblock Lectures in Mathematics ETH Z\"urich. Birkh\"auser Verlag, Basel,
  1996.
\newblock Based on lecture notes by Norbert Hungerb{\"u}hler.

\bibitem[Sh12]{thesis}
B. Sharp.
\newblock Compensation phenomena in geometric partial differential equations.
\newblock http://wrap.warwick.ac.uk/50026/
\newblock Ph.D thesis, University of Warwick 2012.

\bibitem[ShTo13]{Sh_To}
B. Sharp, P. Topping.
\newblock Decay estimates for {R}ivi\`ere's equation, with applications to
  regularity and compactness.
\newblock {\em Trans. Amer. Math. Soc.}, 365(5):2317--2339, 2013.

\bibitem[ShZhu13]{ben_zhu}
B. Sharp, M. Zhu.
\newblock Regularity at the free boundary for Dirac-harmonic maps from surfaces.
\newblock arXiv:1306.4260
\newblock Preprint, 2013.

\bibitem[Ste70]{stein_singular}
Elias~M. Stein.
\newblock {\em Singular integrals and differentiability properties of
  functions}.
\newblock Princeton Mathematical Series, No. 30. Princeton University Press,
  Princeton, N.J., 1970.
  
\bibitem[WaXu09]{WX} C. Wang, D. Xu. Regularity of Dirac-harmonic maps.
\emph{Int. Math. Res. Not. IMRN} 2009, no. 20, 3759--3792.

\bibitem[Wen69]{wente}
Henry~C. Wente.
\newblock An existence theorem for surfaces of constant mean curvature.
\newblock {\em J. Math. Anal. Appl.}, 26:318--344, 1969.

\bibitem[Zhu13]{Zhu} M. Zhu. Regularity for harmonic maps into certain Pseudo-Riemannian manifolds. \emph{J. Math. Pures Appl.}  99 (2013), no. 1, 106--123. 


\end{thebibliography}

\end{document}